\documentclass[11pt]{amsart}
\usepackage{amssymb}
\usepackage[all]{xy}

\newtheorem{theorem}{Theorem}[section]
\newtheorem{proposition}[theorem]{Proposition}
\newtheorem{lemma}[theorem]{Lemma}
\newtheorem{corollary}[theorem]{Corollary}

\theoremstyle{definition}
\newtheorem{definition}[theorem]{Definition}
\newtheorem{example}[theorem]{Example}
\newtheorem{remark}[theorem]{Remark}

\makeatletter
\theoremstyle{theorem}%
\newtheorem*{rep@theorem}{\rep@title}
\newcommand{\newreptheorem}[2]{%
\newenvironment{rep#1}[1]{%
 \def\rep@title{#2 \ref{##1}}%
 \begin{rep@theorem}}%
 {\end{rep@theorem}}}
\makeatother

\newreptheorem{theorem}{Theorem}
\newreptheorem{corollary}{Corollary}
\newreptheorem{lemma}{Lemma}

\numberwithin{equation}{section} 
\numberwithin{figure}{section}
\numberwithin{table}{section}

\newcommand{\Z}{\mathbb{Z}}
\newcommand{\Q}{\mathbb{Q}}
\newcommand{\R}{\mathbb{R}}
\newcommand{\C}{\mathrm{C}}

\renewcommand{\SS}{\mathfrak{S}}
\newcommand{\A}{\mathfrak{A}}
\newcommand{\Aut}{\operatorname{Aut}}
\newcommand{\Out}{\operatorname{Out}}
\newcommand{\Inn}{\operatorname{Inn}}

\newcommand{\Ad}{\operatorname{Ad}}
\newcommand{\id}{\operatorname{id}}
\newcommand{\ab}{\mathrm{ab}}
\newcommand{\Ker}{\operatorname{Ker}}

\newcommand{\ang}[1]{\left\langle{#1}\right\rangle}

\begin{document}

\title[Knot groups in lens spaces and those in $S^3$]
{An explicit relation between knot groups in lens spaces and those in $S^3$}
\author[Y. Nozaki]{Yuta Nozaki}
\address{Graduate School of Mathematical Sciences, the University of Tokyo \\
3-8-1 Komaba, Meguro-ku, Tokyo, 153-8914 \\
Japan}
\email{nozaki@ms.u-tokyo.ac.jp}

\maketitle

\begin{abstract}
 For a cyclic covering map $(\Sigma,K) \to (\Sigma',K')$ between two pairs of a 3-manifold and a knot each, we describe the fundamental group $\pi_1(\Sigma \setminus K)$ in terms of $\pi_1(\Sigma' \setminus K')$.
 As a consequence, we give an alternative proof for the fact that certain knots in $S^3$ cannot be represented as the preimage of any knot in a lens space, which is related to free periods of knots.
 In our proofs, the subgroup of a group $G$ generated by the commutators and the $p$th power of each element of $G$ plays a key role.
\end{abstract}

\setcounter{tocdepth}{1} 
\tableofcontents

%%%%%%%%%%
\section{Introduction}\label{sec:Intro}
Let $K$ be a torus knot and $p \geq 2$ an integer.
Hartley~\cite{Har81} introduced an argument that helps decide whether $K$ can be represented as the preimage of a knot in the lens space $L(p,q)$ for some $q$.
For example, the trefoil $T_{3,2}$ does not appear as the preimage of any knot in $L(2,1) \cong \mathbb{R}P^3$.
The Alexander polynomial of torus knots was used in his proof.

Let $K$ be a knot such that the outer automorphism group $\Out(\pi_1(S^3 \setminus K))$ is trivial.
For instance, $9_{32}$, $9_{33}$ and 24 more prime knots with 10 crossings (and their mirror images) satisfy this condition (see Kawauchi~\cite[Appendix~F.2]{Kaw96}, Kodama-Sakuma~\cite[Table~3.1]{KoSa92}).
Then it follows from Borel's theorem (see Conner-Raymond~\cite[Theorem~3.2]{CoRa72}) that $K$ cannot be represented as the preimage of any knot in any lens space (see \cite[Theorems~10.6.2 and 10.6.6(1)]{Kaw96}).

The purpose of this paper is to deduce the above facts from a single result.
Let $G$ be a group and $p$ a positive integer which is not necessarily prime.
We introduce the subgroup $\C^p(G)$ generated by the commutators and the $p$th power of each element of $G$.
The case $p=2$ was studied, for example, in Sun~\cite{Sun79} and Haugh-MacHale~\cite{HaMa97}, from a purely algebraic point of view.
Moreover, the case where $p$ is prime can be found in Stallings~\cite{Sta65} and Cochran-Harvey~\cite{CoHa08b} (see Remark~\ref{rem:ConferCp} and Appendix~\ref{sec:RepeatCp}).
We apply $\C^p$ to the fundamental group of the complement of a knot in a 3-manifold.
Let $\pi\colon \Sigma \to \Sigma'$ be a $p$-fold cyclic cover, where $\Sigma$ is an integral homology 3-sphere.
The main result of this paper is the following.

\begin{theorem}\label{thm:CovSpCp}
 Let $K'$ be a knot in $\Sigma'$ with preimage a knot $K$.
 Then the image of $\pi_\ast\colon \pi_1(\Sigma\setminus K) \rightarrowtail \pi_1(\Sigma'\setminus K')$ coincides with $\C^p(\pi_1(\Sigma'\setminus K'))$.
\end{theorem}

Considering the case of $\Sigma=S^3$ and $\Sigma'=L(p,q)$, we obtain the following corollaries.

\begin{corollary}\label{cor:Torus}
 Let $m,n,p \in \Z_{\geq2}$ with $\gcd(m,n)=1$.
 Then there exists an integer $q$ and a knot $K'$ in $L(p,q)$ such that $\pi^{-1}(K')$ is isotopic to the torus knot $T_{m,n}$ or its mirror image if and only if $\gcd(mn,p)=1$.
\end{corollary}

Furthermore, Chbili~\cite{Chb03} determined the possible values of the above $q$.
We also observe it in Proposition~\ref{prop:Chb03} via Corollary~1.

\begin{corollary}\label{cor:Out}
 A knot $K$ in $S^3$ with $\Out(\pi_1(S^3 \setminus K))=1$ cannot be represented as the preimage of any knot in any lens space.
\end{corollary}

Another proof for a special case of Corollary~2 is given in Appendix~\ref{subsec:BraidGroup}.
When a knot $K$ is represented as a preimage, it is still unknown whether $K'$ is uniquely determined.
Sakuma~\cite{Sak86}, Boileau and Flapan~\cite{BoFl87} independently proved that if the preimages of oriented knots $K_0'$, $K_1'$ are the same prime knot, then there exists a diffeomorphism $(L(p,q),K_0') \to (L(p,q),K_1')$ preserving the orientations both of $L(p,q)$ and $K_i$.
Note that $K_0'$, $K_1'$ are not necessarily ambient isotopic to each other.
They also showed that the same is true for a composite knot under a condition regarding ``slopes''.
Furthermore, Manfredi~\cite{Man14} recently constructed two knots in a lens space such that they are not ambient isotopic to each other but their preimages are the unknot.

Our work is closely related to periodic diffeomorphisms of $(S^3,K)$ without fixed points.
Indeed, such a diffeomorphism of period $p$ induces a cyclic cover $(S^3,K) \to (L(p,q),K')$ for some $q$, and vice versa (see \cite[Table~1.2]{AFW15} for instance).
Hartley~\cite{Har81} gave a list of possible free periods of prime knots up to 10 crossings.

\subsection*{Acknowledgments}
The author would like to express his gratitude to Takuya Sakasai for his various suggestions.
Also, he would like to thank Makoto Sakuma for pointing out some mistakes in a first draft of the paper and for answering several questions about the symmetry of knots.
Finally, the author wishes to express his thanks to the referees for their careful reading of the manuscript and for their useful comments.
This work was supported by the Program for Leading Graduate Schools, MEXT, Japan.

%%%%%%%%%%%%
\section{Lemmas in group theory and knot theory}\label{sec:Lem}
First of all, we set up notation and prepare basic lemmas that will be used in Section~\ref{sec:ThmCor}.

\subsection{Notation}
For a group $G$, let $\Aut(G)$ denote the automorphism group of $G$ and $\Inn(G)$ the subgroup generated by $\Ad_g$'s, where $\Ad_g$ is defined by $\Ad_g(h) = ghg^{-1}$.
The outer automorphism group $\Out(G)$ is defined by $\Out(G) = \Aut(G)/\Inn(G)$.
Let $X$ be a topological space and $G$ as above.
We set $H_\ast(X):=H_\ast(X;\Z)$ and $H_\ast(G):=H_\ast(G;\Z)$ with the trivial action of $G$ on $\Z$.
For a knot $K$ in $S^3$, let $G(K)$ denote the fundamental group of the complement of $K$, that is, $G(K)=\pi_1(S^3\setminus K)$.
Notice that, we are simplifying the notation by omitting the base point.
Let $\Z_n$ denote the cyclic group of order $n$.

\subsection{Definition of $\C^p$ and related lemmas}
We first introduce an important subgroup and give some examples.

\begin{definition}
 For a group $G$ and a positive integer $p$, let $\C^p(G)$ (or $\C^pG$ for short) denote the subgroup of $G$ generated by the set $\{g^p\mid g\in G\}\cup\{[g,h]\mid g,h \in G\}$, where $[g,h]:=ghg^{-1}h^{-1}$.
 A group $G$ is called a \emph{$\C^p$-group} if there is a group $G'$ such that $G \cong \C^p(G')$.
\end{definition}

Note that $\C^1(G)=G$ and $\C^2(G)=\ang{\{g^2\mid g\in G\}}$.

\begin{remark}\label{rem:ConferCp}
 The subgroup $\C^2(G)$ is denoted by $G^2$ in \cite{Sun79} and by $S(G)$ in \cite{HaMa97}.
 For a prime $p$, $\C^p(G)$ coincides with the first term of the $p$-lower central series in \cite{Sta65} and with the first term of the derived $p$-series in \cite{CoHa08b}.
\end{remark}

\begin{example}\label{ex:CpSymGrp}
 Let us determine the subgroup $\C^p\SS_n$ of the $n$th symmetric group $\SS_n$.
 For $p$ odd, $\C^p\SS_n$ contains all transpositions, and thus $\C^p\SS_n = \SS_n$.
 In the case where $p$ is even, we conclude $\C^p\SS_n = \A_n$ since $\A_n=[\SS_n,\SS_n] \leq \C^p\SS_n \leq \A_n$, where $\A_n$ denotes the $n$th alternating group.
\end{example}

We next state two basic lemmas.

\begin{lemma}\label{lem:ElemLem}
 For any group $G$, the subgroup $\C^p(G)$ is characteristic, and the quotient group $G/\C^p(G)$ is abelian.
 Moreover, if $G$ is finitely generated, then $G/\C^p(G)$ is isomorphic to a finite direct sum of finite cyclic groups whose orders are divisors of $p$.
\end{lemma}

\begin{proof}
 Since we have $kg^pk^{-1} = (kgk^{-1})^p$ and $k[g,h]k^{-1} = [kgk^{-1},khk^{-1}]$ for $g,h,k \in G$, the subgroup $\C^p(G)$ is characteristic.
 The latter statement follows from the fundamental theorem for finitely generated abelian groups.
\end{proof}

Some simple examples of $G/\C^p(G)$ are listed in Table~\ref{tab:G/CpG}.

\begin{table}[h]
 \center
$\begin{array}{c|c}
 G & G/\C^p(G) \\\hline\hline
 (\Z_n,+) & \Z_{\gcd(n,p)} \\\hline 
 (\Q_{>0},\times) & (\Z_p)^{\oplus\{\text{prime numbers}\}} \\\hline
 (\R_{>0},\times) & 0 
\end{array}$
 \caption{Examples of $G/\C^p(G)$. Note that $\gcd(0,p)=p$.}
 \label{tab:G/CpG}
\end{table}

\begin{lemma}
 For groups $G$ and $H$, the following hold.
\begin{enumerate}
 \item $\C^p(G\times H) = \C^p(G)\times\C^p(H)$.
 \item A homomorphism $f\colon G\to H$ induces the homomorphism $\C^p(G)\to\C^p(H)$ denoted by $\C^p(f)$.
 Moreover, if $f$ is surjective \textup{(}resp.\ injective\textup{)}, then $\C^p(f)$ is also surjective \textup{(}resp.\ injective\textup{)}.
\end{enumerate}
\end{lemma}

\begin{proof}
 (1) The inclusion ``$\leq$'' is clear.
 The opposite inclusion is shown by the equality $(g,h) = (g,1)(1,h)$.
 
 (2) We only prove the surjectivity of $\C^p(f)$.
 Let $h' \in \C^p(H)$.
 By definition, $h'$ is of the form $h_1^p[h_2,h_3]\cdots$.
 When $f$ is surjective, there is $g_i \in G$ such that $f(g_i) = h_i$ for each $i$, and then $f(g_1^p[g_2,g_3]\cdots) = h'$.
\end{proof}

The following lemma is a refinement of the well-known fact \cite[13.5.8]{Rob96} for a complete group $H$, where $H$ is said to be \emph{complete} if $\Out(H)$ is trivial.

\begin{lemma}\label{lem:Rob96}
 Let $G$ and $H$ be groups such that $H \vartriangleleft G$ and ${\Ad_g}|_H \in \Inn(H)$ for any $g\in G$.
 Then the sequence of groups
 \[1 \to Z(H) \xrightarrow{\phi} H\times C_G(H) \xrightarrow{\psi} G \to 1\]
 is exact, where $C_G(H)$ denotes the centralizer of $H$ in $G$, and $\phi(h):=(h,h^{-1})$, $\psi(h,g):=hg$.
\end{lemma}

\begin{lemma}[see {\cite[Theorem~1]{HaMa97}}]\label{lem:HaMa97}
 Let $G$, $H$ be as in Lemma~\ref{lem:Rob96} and suppose $\C^p(G)=H$, $Z(H)=1$.
 Then $\C^p(H) = H$.
\end{lemma}

\begin{proof}
 Set $K:=C_G(H)$.
 The isomorphism $\psi$ induces $\C^p(\psi)\colon \C^p(H)\times\C^p(K) \xrightarrow{\cong} \C^p(G)=H$.
 Since $\C^p(\psi)$ sends $\C^p(H)\times\{1\}$ to $\C^p(H)$, we have $\C^p(K) \xrightarrow{\cong} H/\C^p(H)$, and hence $Z(\C^p(K))=\C^p(K)$.
 On the other hand, $\C^p(\psi)$ gives $Z(\C^p(H))\times Z(\C^p(K)) \xrightarrow{\cong} Z(H)=1$.
 Therefore $\C^p(K)=1$, and we conclude that $\C^p(\psi)\colon \C^p(H) \xrightarrow{\cong} H$.
\end{proof}

\begin{example}\label{ex:IsSnCp}
 Let us see whether $\SS_n$ is a $\C^p$-group.
 Since $\SS_1$ and $\SS_2$ are obviously $\C^p$-groups for any $p$, we suppose $n\geq3$.
 If $p$ is odd, then $\SS_n$ is still a $\C^p$-group by Example~\ref{ex:CpSymGrp}.
 Suppose that $p$ is even.
 If $n\neq6$, then $\SS_n$ is not a $\C^p$-group by Lemma~\ref{lem:HaMa97} and the fact that $\SS_n$ is complete except $n=2,6$.
 Furthermore, $\SS_6$ is not a $\C^p$-group (see Appendix~\ref{subsec:S6}).
\end{example}

\begin{lemma}\label{lem:FreeAbCp}
 Let $r$ be a non-negative integer.
 For a group $G$ whose abelianization $G^\ab$ is isomorphic to $\Z^{\oplus r}$, the quotient group $G/\C^p(G)$ is isomorphic to $(\Z_p)^{\oplus r}$.
\end{lemma}

\begin{proof}
 Since the codomain of the projection $G \twoheadrightarrow G/\C^p G$ is abelian, it factors through $G^\ab \cong \Z^{\oplus r}$.
 Hence, $G/\C^p G \cong \bigoplus_{i=1}^r \Z_{p_i}$ for some $p_i \mid p$.
 The abelianization $G \twoheadrightarrow \Z^{\oplus r}$ induces a surjection $G/\C^p G \twoheadrightarrow \Z^{\oplus r}/\C^p(\Z^{\oplus r})=(\Z_p)^{\oplus r}$.
 It follows that $p_i=p$ for $i=1,\dots,r$.
\end{proof}

\begin{remark}
 If $p$ is prime, then Lemma~\ref{lem:FreeAbCp} directly follows from \cite[Theorem~3.4]{Sta65} or \cite[Lemma~2.3]{CoHa08b}.
\end{remark}

\subsection{Links in a rational homology 3-sphere}

The following lemma is a special case of \cite[Theorem~2.1]{HaMu78}.

\begin{lemma}\label{lem:SingHom}
 Let $L'$ be an $r$-component link in a rational homology 3-sphere $\Sigma'$ such that $H_1(L') \xrightarrow{\mathrm{incl}_\ast} H_1(\Sigma')$ is surjective.
 Then the following holds:
 \[H_\ast(\Sigma'\setminus L') \cong
\begin{cases}
 \Z & \text{if $\ast=0$}, \\
 \Z^{\oplus r} & \text{if $\ast=1$}, \\
 \Z^{\oplus(r-1)} & \text{if $\ast=2$}, \\
 0 & \text{otherwise}.
\end{cases}\]
\end{lemma}

\begin{remark}\label{rem:Meridian}
 If $H_1(\Sigma') \cong \Z_p$ and $r=1$, then a meridian of a knot $L'$ corresponds to $\pm p \in\Z \cong H_1(\Sigma'\setminus L')$.
\end{remark}

%%%%%%%%%%%%%%%%%%%%%%%%
\section{Main theorem and corollaries}\label{sec:ThmCor}
Recall that we are interested in knowing whether a knot $K$ in $S^3$ is represented as the preimage of a knot $K'$ in a lens space.

\subsection{Application of $\C^p$ to covering spaces}
Let $\pi\colon \Sigma \to \Sigma'$ be a $p$-fold cyclic cover, where $\Sigma$ is an integral homology 3-sphere.
One can check that $H_\ast(\Sigma') \cong H_\ast(L(p,q))$.

\begin{reptheorem}{thm:CovSpCp}
 Let $K'$ be a knot in $\Sigma'$ with preimage a knot $K$.
 Then the image of $\pi_\ast\colon \pi_1(\Sigma\setminus K) \rightarrowtail \pi_1(\Sigma'\setminus K')$ coincides with $\C^p(\pi_1(\Sigma'\setminus K'))$.
\end{reptheorem}

\begin{proof}
 Set $G:=\pi_1(\Sigma\setminus K)$ and $G':=\pi_1(\Sigma'\setminus K')$.
 By Lemma~\ref{lem:SingHom} and Remark~\ref{rem:Connected}, the surjection $(G')^\ab \twoheadrightarrow G'/\C^pG'$ implies $G'/\C^pG' \cong \Z_{p'}$ for some $p' \mid p$.
 
 On the other hand, the covering map $\pi$ induces an exact sequence
 \[1 \to G \xrightarrow{\pi_\ast} G' \to \Z_{p} \to 1.\]
 Using $\C^pG \leq \pi_\ast(G) \leq G'$, we have
 \[p' = [G':\C^pG'] = [G':\pi_\ast(G)][\pi_\ast(G):\C^pG'] \geq p.\]
 It follows that $p'=p$, and thus $\pi_\ast(G) = \C^pG'$.
\end{proof}

\begin{remark}\label{rem:Connected}
 Let $r$ be the number of connected components of $\pi^{-1}(K')$ and $k'$ the positive integer such that $p/k' \in \Z$ is the order of $[K'] \in H_1(\Sigma')$.
 Fix the isomorphism $H_1(\Sigma') \cong \Z_p$ so that $[K']$ corresponds to $\overline{k'}$
 Then we have $r=k'$.
 Indeed, we first obtain the canonical isomorphism $H_1(\Sigma') \cong \pi_1(\Sigma')/\pi_\ast\pi_1(\Sigma)$ since the projection $\pi_1(\Sigma') \twoheadrightarrow \pi_1(\Sigma')/\pi_\ast\pi_1(\Sigma) \cong \Z_p$ factors through $\pi_1(\Sigma')^\ab$.
 Choose a connected component $K$ of $\pi^{-1}(K')$ and a base point $\ast \in K$, and set $\ast':=\pi(\ast)$.
 It follows from $\pi(K)={K'}^{p/r}$ as based loops and $[K]=0$ that we have $\frac{p}{r}\times\overline{k'}=0$, and thus $r \mid k'$.
 On the other hand, the equality $\frac{p}{k'}\times[K']=0$ implies that ${K'}^{p/k'} \in \pi_1(\Sigma',\ast')$ belongs to $\pi_\ast\pi_1(\Sigma)$, and hence ${K'}^{p/k'}$ lifts to $\Sigma$.
 Then we conclude $\frac{p}{r} \mid \frac{p}{k'}$, namely $k' \mid r$.
 
 In particular, $\pi^{-1}(K')$ is connected if and only if $[K']$ generates $H_1(\Sigma)$.
 The case of $\Sigma=S^3$ is discussed in \cite[Proposition~2.2]{Man14}.
\end{remark}

As mentioned in Section~\ref{sec:Intro}, the following corollary is already known.
However, using Theorem~\ref{thm:CovSpCp}, we give an alternative proof of it.

\begin{repcorollary}{cor:Out}
 A knot $K$ in $S^3$ with $\Out(G(K))=1$ cannot be represented as the preimage of any knot in any lens space.
\end{repcorollary}

\begin{proof}
 Since $\Out(G(T_{m,n})) \cong \Z_2$ as it was proven in \cite{Sch24}, then $K$ is not a torus knot.
 Hence, we conclude that $Z(G(K))=1$ by \cite{BuZi66}, and thus $G(K)$ is complete.
 
 Assume that there exists a knot $K'$ in $L(p,q)$ whose preimage is isotopic to $K$.
 Then we have $G(K) = \C^p(\pi_1(L(p,q)\setminus K'))$ by Theorem~\ref{thm:CovSpCp}.
 Since $G(K)$ is complete, by Lemma~\ref{lem:HaMa97}, we conclude $\C^pG(K) = G(K)$.
 However, this contradicts Lemma~\ref{lem:FreeAbCp}.
\end{proof}

\subsection{Application of Theorem~\ref{thm:CovSpCp} to torus knots}
The aim of this subsection is to prove the next corollary of Theorem~\ref{thm:CovSpCp}.

\begin{repcorollary}{cor:Torus}[{\cite[Theorem~3.1]{Har81}}]
 Let $m,n,p \in \Z_{\geq2}$ with $\gcd(m,n)=1$.
 There exists an integer $q$ and a knot $K'$ in $L(p,q)$ such that $\pi^{-1}(K')$ is isotopic to the torus knot $T_{m,n}$ or its mirror image if and only if $\gcd(mn,p)=1$.
\end{repcorollary}

\begin{lemma}\label{lem:HomOfGrp}
 Let $m,n,p$ be positive integers and $G$ a group satisfying the three conditions
\begin{enumerate}
 \item $\C^p(G) \cong \Z_m\ast\Z_n$,
 \item $G/\C^p(G) \cong \Z_p$,
 \item $|G^\ab|=mnp$.
\end{enumerate}
 Then $\gcd(mn,p)=1$.
\end{lemma}

\begin{proof}
 We first consider the five-term exact sequence
 \[0=H_2(\Z_p) \to (\Z_m\oplus\Z_n)_{\Z_p} \to G^\ab \to \Z_p \to 0\]
 for the short exact sequence $1 \to \Z_m\ast\Z_n \to G \to \Z_p \to 1$ coming from (1) and (2).
 By (3), we conclude that $(\Z_m\oplus\Z_n)_{\Z_p} = \Z_m\oplus\Z_n$.
 
 Set $H:=\Z_m\ast\Z_n$.
 Let $g \in G$ and we show that ${\Ad_g}|_{H} \in \Inn(H)$.
 It is known that $\Aut(H)$ is generated by $\Inn(H)$ and $\{\phi_{i,j} \mid i \in (\Z_m)^\times,\ j \in (\Z_n)^\times\}$, where $\phi_{i,j}$ is defined by $\phi_{i,j}(a)=a^i$, $\phi_{i,j}(b)=b^j$.
 Therefore, ${\Ad_g}|_{H} = {\Ad_w}\circ\phi_{i,j}$ for some $1\leq i<m$, $1\leq j<n$ and $w \in H$, and then the induced automorphism $({\Ad_g}|_{H})^\ab$ of $\Z_m\oplus\Z_n$ sends $(1,0)$, $(0,1)$ to $(i,0)$, $(0,j)$ respectively.
 On the other hand, we have $(\Ad_g)^\ab = \id_{G^\ab}$, and hence $i=1$ and $j=1$.
 
 It follows from ${\Ad_g}|_{H} \in \Inn(H)$, $Z(H)=1$ and (1) that one can use Lemma~\ref{lem:HaMa97} and conclude $\C^p(H)=H$.
 Finally, since the abelianization induces a surjection $H/\C^p(H) \twoheadrightarrow (\Z_m\oplus\Z_n)/\C^p(\Z_m\oplus\Z_n) = \Z_{\gcd(m,p)}\oplus\Z_{\gcd(n,p)}$, we get $\gcd(mn,p)=1$.
\end{proof}

\begin{remark}
 The three conditions in Lemma~\ref{lem:HomOfGrp} also imply that $H_\ast(G) \cong H_\ast(\Z_m\ast\Z_n\ast\Z_p)$.
 Indeed, it follows from ${\Ad_g}|_{H} \in \Inn(H)$ that the action of $G$ on $H_\ast(\Z_m\ast\Z_n)$ is trivial, and so is the action of $\Z_p$ on $H_\ast(\Z_m\ast\Z_n)$.
 The conclusion $\gcd(mn,p)=1$ of Lemma~\ref{lem:HomOfGrp} implies that the $E^2$-term of the Lyndon-Hochschild-Serre spectral sequence for $1 \to \Z_m\ast\Z_n \to G \to \Z_p \to 1$ is as follows:
 \[E^2_{s,t}=H_s(\Z_p;H_t(\Z_m\ast\Z_n))=
\begin{cases}
 \Z & \text{if $s=t=0$,} \\
 \Z_m\oplus\Z_n & \text{if $s=0$, $t\geq1$ is odd,} \\
 \Z_{p} & \text{if $s\geq1$ is odd, $t=0$,} \\
 0 & \text{otherwise.}
\end{cases}\]
 Therefore, all differentials $d^r$ are zero for $r\geq2$.
 Using $\gcd(mn,p)=1$ again, we have $H_\ast(G) \cong H_\ast(\Z_m\ast\Z_n\ast\Z_p)$.
\end{remark}

\begin{proof}[Proof of Corollary~\ref{cor:Torus}]
 If $\gcd(mn,p)=1$, then a construction of a desired knot $K'$ was given in \cite[Theorem~3.1]{Har81}.
 
 Suppose that there exists $K'$ as in the statement, and set $\pi_1':=\pi_1(L(p,q)\setminus K')$.
 The covering map $\pi\colon S^3\setminus T_{m,n} \to L(p,q)\setminus K'$ induces the exact sequence
 \[1 \to G(T_{m,n})=\ang{a,b \mid a^m=b^n} \xrightarrow{\pi_\ast} \pi_1' \to \Z_p \to 1.\]
 Since the center $Z(G(T_{m,n})) = \ang{a^m} = \Z$ is characteristic in $G(T_{m,n})$, the subgroup $N:=\pi_\ast(\ang{a^m})$ of $\pi_1'$ is normal.
 We deduce the exact sequence
\begin{align*}
 1 \to \ang{a,b \mid a^m=1=b^n} \xrightarrow{\pi_\ast} \pi_1'/N \to \Z_p \to 1
\end{align*}
 from the third isomorphism theorem.
 
 Here, we define the group $G$ by $G:=\pi_1'/N$.
 Then Theorem~\ref{thm:CovSpCp} shows that
 \[\C^p(G) = \C^p(\pi_1')/N = G(T_{m,n})/\ang{a^m} = \Z_m\ast\Z_n.\]
 Therefore, $G$ satisfies (1) and (2) in Lemma~\ref{lem:HomOfGrp}, and hence it suffices to prove (3).
 By Lemma~\ref{lem:SingHom}, the five-term exact sequence for $1 \to N \to \pi_1' \to G \to 1$ is as follows:
 \[H_2(\pi_1') \to H_2(G) \to \Z_G \xrightarrow{\phi} \Z \to H_1(G) \to 0.\]
 Since the generator $a^m$ of $N$ is decomposed as $\left(\prod_{i=1}^{n}g_i\mu g_i^{-1}\right)^m$ in $\pi_1'$ for a meridian $\mu$ of $K'$ and some $g_i \in \pi_1'$, we conclude $\phi(\overline{1})=mnp$ by Remark~\ref{rem:Meridian}.
 We thus get $H_1(G)=\Z_{mnp}$ (and $\Z_G = \Z$).
\end{proof}

The above proof is not enough to assert that $G(T_{m,n})$ is not a $\C^p$-group when $\gcd(mn,p)\neq1$.
We finally discuss integers $q$ in Corollary~\ref{cor:Torus}.

\begin{proposition}[{\cite[Remark~3.7]{Chb03}}] \label{prop:Chb03}
 Let $m,n,p$ be integers with $|m|,|n|,p \geq2$ and $\gcd(m,n)=1$.
 Suppose that the torus knot $T_{m,n}$ is isotopic to the preimage of a knot in $L(p,q)$ for some $q$.
 Then $q \equiv m(1-pp^\ast)/n \mod p$, where $p^\ast$ is an integer satisfying $pp^\ast \equiv 1 \mod n$.
\end{proposition}

\begin{proof}
 We first remark that $\gcd(mn,p)=1$ by Corollary~\ref{cor:Torus}.
 Hence, $p^\ast$ exists and $\gcd(p,m(1-pp^\ast)/n)=\gcd(p,1-pp^\ast)=1$ holds.
 
 Since $T_{m,n}$ is a prime knot, by \cite[Theorem~5]{Sak86} or \cite[Theorem~1]{BoFl87}, $q$ is uniquely determined up to taking its inverse in $\Z_p$.
 (Recall that $L(p,q_0) \cong L(p,q_1)$ as oriented manifolds if and only if $q_0 \equiv q_1^{\pm1} \mod p$.)
 Therefore, it suffices to see that there exists a knot in $L(p,q)$ whose preimage is $T_{m,n}$ for $q=m(1-pp^\ast)/n$.
 Set $q:=m(1-pp^\ast)/n$.
 Then we see that $p \mid m-nq$.
 Hence, using a suitable $n$-braid (see \cite[Proposition~3.3]{Man14}), one constructs a desired knot in $L(p,q)$.
\end{proof}

%%%%%%%%%%%%%
\appendix
\section{Braid groups and symmetric groups}\label{sec:SymGrp}

\subsection{Braid groups}\label{subsec:BraidGroup}
In this subsection, we discuss the question of whether the $n$th braid group is a $\C^p$-group.
Since $B_3 \cong G(T_{3,2})$, the answer gives another proof of Corollary~\ref{cor:Torus} for the trefoil $T_{3,2}$.
First note that a quotient group of a $\C^p$-group $G$ is not a $\C^p$-group in general.
However, the following lemma shows that if $H \vartriangleleft G$ is characteristic, then $G/H$ is also a $\C^p$-group.

\begin{lemma}[see {\cite[Theorem~1]{Sun79}}]\label{lem:Sun79}
 Let $G$ be a $\C^p$-group and $f\colon G \twoheadrightarrow H$ a surjective homomorphism whose kernel is a characteristic subgroup of $G$.
 Then $H$ is also a $\C^p$-group.
\end{lemma}

\begin{proof}
Suppose $\C^p(G')=G$ for some group $G'$.
Then we have
$$\C^p(G'/\Ker f) = \C^p(G')/(\Ker f \cap \C^p(G')) = G/\Ker f \cong H.$$
Hence, $H$ is a $\C^p$-group.
\end{proof}

\begin{corollary}\label{cor:Braid}
The $n$th braid group $B_n$ is not a $\C^p$-group for $p$ even.
\end{corollary}

\begin{proof}
Since the $n$th pure braid group $P_n:=\Ker[B_n \twoheadrightarrow \SS_n]$ is characteristic (\cite[Theorem 3]{Art47}), by Lemma~\ref{lem:Sun79}, it suffices to prove that $\SS_n$ is not a $\C^p$-group for $p$ even.
Finally, Example~\ref{ex:IsSnCp} completes the proof.
\end{proof}

It is not known to us whether $B_n$ is a $\C^p$-group for $p$ odd.

\subsection{The 6th symmetric group}\label{subsec:S6}
In this subsection, we prove that $\SS_6$ is not a $\C^p$-group for $p$ even (see Example~\ref{ex:CpSymGrp}).
The following argument is based on \cite{HaMa97}.
Let $\iota$ denote the homomorphism $\SS_6 \rightarrowtail \Aut(\SS_6)$ defined by $\iota(\sigma):=\Ad_\sigma$.

\begin{lemma}[see {\cite[p.\ 125]{HaMa97}}]\label{lem:HaMa97S6}
 For $p$ even, $\C^p(\Aut(\SS_6)) = \iota(\A_6)$.
\end{lemma}

\begin{proof}
 It follows from $[\SS_6,\SS_6]=\A_6$ that $\C^p(\Aut(\SS_6)) \geq \iota(\A_6)$.
 One has to show that $\phi^p$ and $[\phi,\phi']$ belong to $\iota(\A_6)$ for any $\phi,\phi' \in \Aut(\SS_6)$.
 We prove only $\phi^p \in \iota(\A_6)$ for $\phi \notin \Inn(\SS_6)$, and the remainder is shown similarly.
 
 It is known that there exists $\psi \in \Aut(\SS_6)\setminus\Inn(\SS_6)$ whose order is 10.
 Since $\Out(\SS_6) \cong \Z_2$, we have $\psi^2=\Ad_\tau$ for some $\tau \in \SS_6$.
 It follows from $\psi^{10}=\id_{\SS_6}$ that $\tau^5=1$, and thus $\tau \in \A_6$.
 Hence, $\phi$ is written as $\phi=\psi\circ\Ad_\sigma$ for some $\sigma \in\SS_6$.
 Then we have
 \[(\psi\circ\Ad_\sigma)^p = (\psi^2\circ\Ad_{\psi^{-1}(\sigma)\sigma})^{p/2} = \Ad_{(\tau\psi^{-1}(\sigma)\sigma)^{p/2}}.\]
 Since $\psi^{-1}(\sigma)\sigma \in \A_6$, we conclude $\phi^p \in \iota(\A_6)$.
\end{proof}

\begin{lemma}[see {\cite[Theorem~3]{HaMa97}}]\label{lem:HaMa97Aut}
 Let $p$ be a positive integer and $G$ a $\C^p$-group.
 Then $\Inn(G) \leq \C^p(\Aut(G))$.
\end{lemma}

\begin{proof}
 Suppose $G=\C^p(G')$ and recall that $G$ is a characteristic subgroup of $G'$.
 The composite map $\Inn(G) \to \Inn(G') \to \Aut(G)$ coincides with the inclusion.
 Since $\Inn(G)=\C^p\Inn(G')$, we have the inclusion $\Inn(G) \hookrightarrow \C^p\Aut(G)$.
\end{proof}

\begin{proposition}
 $\SS_6$ is not a $\C^p$-group for $p$ even.
\end{proposition}

\begin{proof}
 Assume that $\SS_6$ is a $\C^p$-group.
 By Lemma~\ref{lem:HaMa97Aut}, one has $\iota(\SS_6) \leq \C^p(\Aut(\SS_6))$, and thus
 \[2 = [\Aut(\SS_6):\iota(\SS_6)] = [\Aut(\SS_6):\C^p{\Aut(\SS_6)}][\C^p{\Aut(\SS_6)}:\iota(\SS_6)].\]
 On the other hand, we see that $\Aut(\SS_6)/\C^p{\Aut(\SS_6)} \twoheadrightarrow \Z_2/\C^p\Z_2 = \Z_2$.
 Hence, $\C^p{\Aut(\SS_6)} = \iota(\SS_6)$.
 This contradicts Lemma~\ref{lem:HaMa97S6}.
\end{proof}

\subsection{Alternating groups}
We finally determine the group $\C^p(\A_n)$.

\begin{proposition}
 For $p \in\Z_{\geq 2}$ and $n \in \Z_{\geq3}$,
 \[\C^p(\A_n) = 
\begin{cases}
 1 & \text{if $n=3$ and $3 \mid p$,} \\
 V_4 & \text{if $n=4$ and $3 \mid p$,} \\
 \A_n & \text{otherwise,}
\end{cases}\]
 where $V_4$ is the Klein four-group.
\end{proposition}

\begin{proof}
 The case of $n=3$ is obvious since $\A_3 = \Z_3$.
 In the case of $n\geq5$, it is well known that $\A_n^\ab = 0$, and thus $\C^{p}\A_n = \A_n$.
 Since $[\A_4,\A_4]=V_4$, $\C^{p}\A_4$ is isomorphic to $V_4$ or $\A_4$.
 Here, $\C^{p}\A_4$ contains cyclic permutations of order three if and only if $3 \nmid p$.
\end{proof}

%%%%%%%%%%%%%%%%%%%%%%
\section{Taking $\C^p$ repeatedly}\label{sec:RepeatCp}
Let us discuss the derived $p$-series of a knot group.
In \cite{CoHa08b}, for a prime $p$, the derived $p$-series $\{G^{(n)}\}_{n\geq0}$ of a group $G$ is defined by 
\[G^{(0)}:= G,\quad G^{(n+1)}:= \C^p(G^{(n)}).\]
We use this notation when $p$ is not only a prime but also a positive integer.
Here, determining the successive quotients $G^{(n)}/G^{(n+1)}$ and the intersection $G^{(\omega)}:=\bigcap_{n\geq0}G^{(n)}$ is a fundamental problem.
In particular, we are interested in the case where $G$ is a knot group or the braid group $B_n$.
Note that the abelianizations of these groups are isomorphic to $\Z$.

\begin{proposition}\label{prop:SuccessiveQuot}
 Let $G$ be a group with $G^\ab \cong \Z^{\oplus r}$.
 Then the abelian group $G^{(n)}/G^{(n+1)}$ surjects onto $\Z_p^{\oplus r}$.
\end{proposition}

\begin{proof}
 The abelianization of $G$ induces the homomorphism $G^{(n)} \twoheadrightarrow (\Z^{\oplus r})^{(n)}$ for each $n$.
 Hence, $G^{(n)}/G^{(n+1)} \twoheadrightarrow (\Z^{\oplus r})^{(n)}/(\Z^{\oplus r})^{(n+1)} \cong \Z_p^{\oplus r}$.
\end{proof}

\begin{remark}
 In contrast, the stability of the (usual) derived series of $G(K)$ depends on a knot $K$ (\cite[Corollary~4.8]{Coc04}, \cite[Proposition~2.7]{Rou02}).
\end{remark}

\begin{proposition}\label{prop:Omega}
 Let $p$ be a prime, $G$ be a finitely generated group with $H_1(G)\cong\Z$.
 Then $G^{(\omega)}$ coincides with the commutator subgroup of $G$.
\end{proposition}

\begin{proof}
 Since the abelianization $\alpha\colon G \twoheadrightarrow \Z$ induces an isomorphism on $H_1(-;\Z_p)$ and a surjection on $H_2(-;\Z_p)$, by \cite[Corollary~4.3]{CoHa08b}, $\alpha$ induces an isomorphism $G/G^{(\omega)} \to \Z/\Z^{(\omega)} = \Z$.
 Therefore, $G^{(\omega)}$ coincides with $\Ker\alpha = [G,G]$.
\end{proof}

\begin{remark}
 Even if $p$ is not prime, by considering a prime factor of $p$, one gets $G^{(\omega)} \leq [G,G]$.
 However, the equality $G^{(\omega)} = [G,G]$ does not hold in general.
 Indeed, we consider the case of $6 \mid p$ and $G=G(3_1)$, where $3_1$ is the trefoil.
 It follows from $\SS_3^{(2)} = \A_3^{(1)} = 1$ that the homomorphism $G(3_1) \cong B_3 \twoheadrightarrow \SS_3$ induces $G(3_1)/G(3_1)^{(\omega)} \twoheadrightarrow \SS_3$.
 Since $\SS_3$ is not abelian, $G(3_1)^{(\omega)}$ is strictly smaller than the commutator subgroup of $G(3_1)$.
 
 The same argument shows that if $G(K)$ surjects onto $G(3_1)$, then the equality does not hold as well.
 Such knots are found, for example, in \cite[Theorem~1.2]{HKMS11}.
\end{remark}

\begin{corollary}
 For $p$ and $G$ as in Proposition~\ref{prop:Omega}, $G^{(n)}/G^{(n+1)}$ is isomorphic to $\Z_p$.
\end{corollary}

\begin{proof}
 By Proposition~\ref{prop:Omega}, we have homomorphisms
 \[G^{(n)}/G^{(n+1)} \hookrightarrow G/G^{(n+1)} \twoheadleftarrow G/G^{(\omega)} \cong \Z.\]
 Thus $G^{(n)}/G^{(n+1)}$ is a cyclic group whose order is a divisor of $p$.
 Finally, Proposition~\ref{prop:SuccessiveQuot} completes the proof.
\end{proof}

\def\cprime{$'$} \def\cprime{$'$}

\end{document}